\documentclass[12pt,a4paper]{amsart}

\usepackage[dvips]{graphicx}

\newtheorem{thm}{Theorem}
\newtheorem{cor}[thm]{Corollary}
\newtheorem{lem}[thm]{Lemma}

\def\figref#1{Figure \ref{fig:#1}}

\def\figs#1#2#3{%
 \begin{figure}[tp]%
  \centerline{\includegraphics[scale=#3]{#1}}%
  \caption{#2}%
  \label{fig:#1}%
 \end{figure}
 }
\usepackage{color}

\begin{document}

\title{How support lines touch an arc}
\author{Wacharin Wichiramala}
\maketitle
\begin{center}Department of Mathematics and Computer Science,\\
Faculty of Science, Chulalongkorn University
\end{center}

\begin{abstract}
We prove that each simple polygonal arc $\gamma$ attains at most two pairs of support lines of given angle difference such that each pair has $s_1<s_2<s_3$ that $\gamma(s_1)$ and $\gamma(s_3)$ are on one such line and $\gamma(s_2)$ is on the other line.
\end{abstract}


\section{Interesting way to look at support lines}

From [WW], to show that a compact set $K$ can cover every unit arc, it suffices to show that $K$ can cover every simple polygonal unit arc. In the work by Coulton and Movshovich [CM], they prove that for a simple polygonal arc $\gamma$, there is a pair of parallel support lines with points $A$, $B$ and $C$ appear on $\gamma$ in this parametric order such that $A$ and $C$ are one such line and $B$ is on the other line. In this work, we will generalize this result to a pair of lines with any given angle different. In addition, we show that for the parallel case, the pair of support lines is unique.

We first let $\gamma$ be a simple polygonal unit arc parametrized by arc length and suppose that $\gamma$ is not straight. Let $P_1,...,P_n$ be the corner points of its hull which appear in this parametric order with parameters $t_1,...,t_n$.
We first consider the multi-valued function
$T$ that tells when (in which parameter) a support line touches $\gamma$ as follows.
For convenient, we write $L_\theta$ for the support line of angle $\theta$ which is the line containing the ray making angle $\theta$ to X-axis and having $\gamma$ on its left side.
For example, $L_0$ is the horizontal support line under $\gamma$.
For each $\theta$, let $T(\theta)=\{\min_{\gamma(s)\in L_\theta}s,\max_{\gamma(s)\in L_\theta}s\}$.
It is clear that $T$ is periodic of period $2\pi$ and rotating $\gamma$ for
angle $\theta$ counter clockwise shifts the graph of $T$ to the right for
$\theta$.
Without loss of generality, we may rotate $\gamma$ so that $P_1$ and $P_2$ are on $L_0$. 
Hence, accordingly $T$ minimizes at 0 with value $t_1$.

Now let us consider the properties of the graph of $T$.
\begin{lem}
The graph of $T$ is composed of horizontal segments at level $t_1,...,t_n$.
\end{lem}
\begin{proof}
A corner of the hull is corresponding to a horizontal segment. A line segment on the boundary of the hull is corresponding to a vertical 2-valued jump.
\end{proof}

Hence $T$ is  a step function with finitely many 2-valued jumps.
The width of each step is $\pi-\theta$ where $\theta$ is the interior angle
of the hull at the corresponding corner of the hull. The different is called
the exterior angle. Thus each width is not zero. Let
$\delta_1$ and $\delta_n$ be the widths corresponding to corners $P_1$ and
$P_n$.

\begin{lem} \label{hor conn}
For each $s\in[0,1]$, the horizontal cross section $T^{-1}(s)$ is either empty or a closed interval of length less than $\pi$.
\end{lem}
\begin{proof}
If not empty, the cross section is connected, not broken, as follows. Suppose $\theta_1$ and $\theta_2$ are in the section with $\theta_1<\theta_2$. Then $L_{\theta_1}$ and $L_{\theta_2}$ touches $P=\gamma(s)$. Hence $P$ is a corner of the hull. Therefore $[\theta_1,\theta_2]$ is a subset of the section.
Let $\theta$ be the exterior angle of the hull at $P$. Thus $0<\theta<\pi$.
Since $\gamma$ is not straight, the 2 angles must be at most $\theta$ apart.
Therefore the different is less than $\pi$.
\end{proof}


Clearly $T$ is not constant. Furthermore, over $[0,2\pi]$, $T$ is
initially monotone increasing from $t_1$ and finally monotone decreasing back to $t_1$.
The next lemma shows that $T$ is not zigzag.
\begin{lem}
$T$ is monotone increasing and then monotone decreasing on $[0,2\pi]$.
\end{lem}
\begin{proof}
Suppose the contrary to get a contradiction that the graph is zigzag. From
the feature of the graph as in Lemma 1, there exist
$0<\theta_1<\theta_2<\theta_3<2\pi$ and $ s_1,s_2,s_3$ that $\{s_i\}=T(\theta_i)$
and $t_1<s_1>s_2<s_3$ as illustrated by \figref{up-down}.
\figs{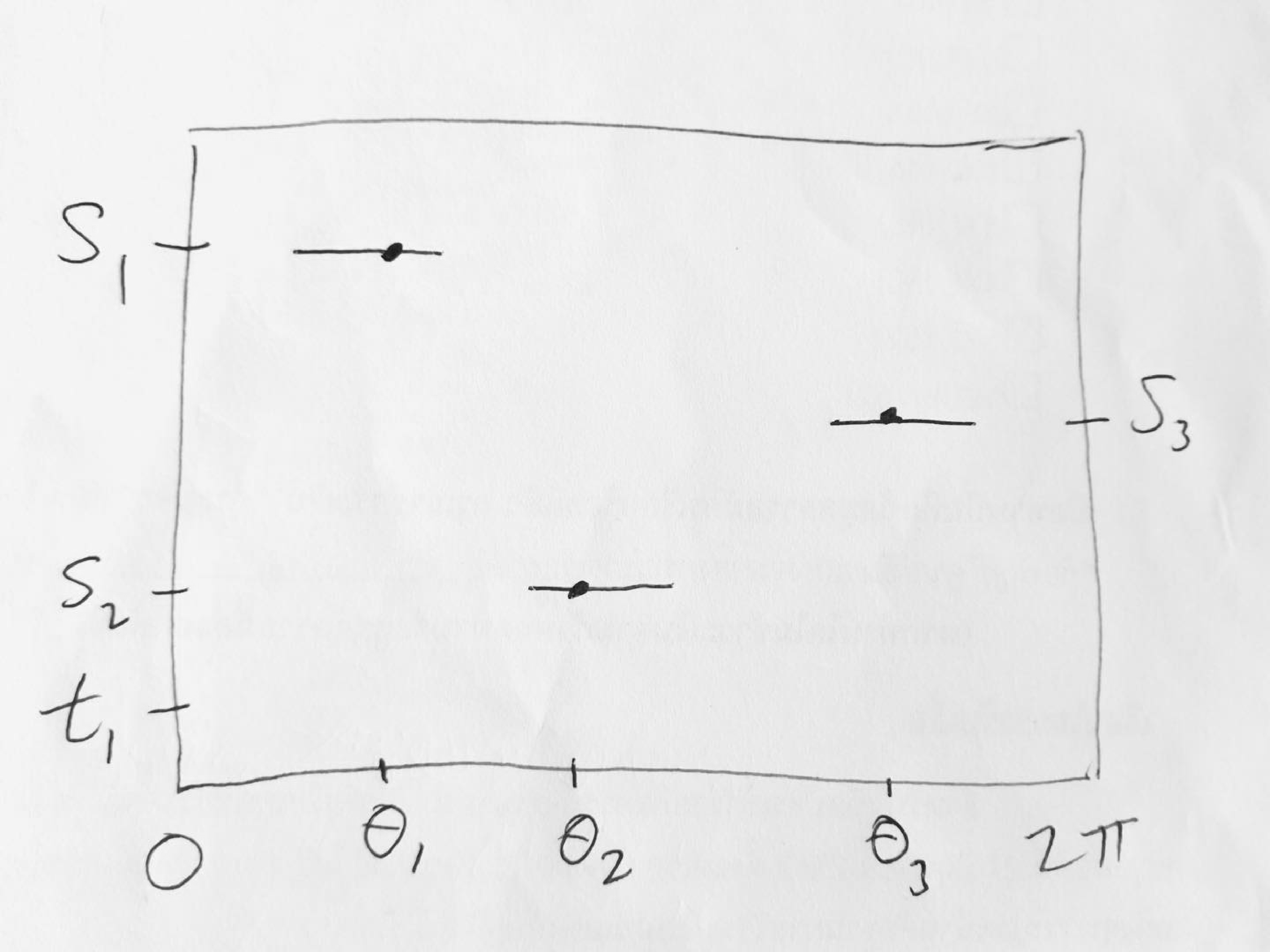}{When the graph of $T$ is zigzag, not just increasing and then decreasing.}{.12}
Firstly we have $t_1\le s_2<s_1$. Next we will show that
$s_2>t_1$. Suppose that $s_2=t_1$. Since $\gamma$ is not on a line, $\gamma(t_1)=P_1$ is a corner of the
hull touched by $L_0$ and $L_{\theta_2}$. Hence $T$ is a constant over $[0,\theta_2]\ni\theta_1$.
Then $T(\theta_1)=\{t_1\}$, a contradiction.
Therefore $t_1<s_2<s_1$.
Since $\theta_1<\theta_3$ and $\gamma$ is not straight, by Lemma 2, we have $s_1\ne s_3$.
Now we have either $s_2<s_1<s_3$ or $s_2<s_3<s_1$.
\figs{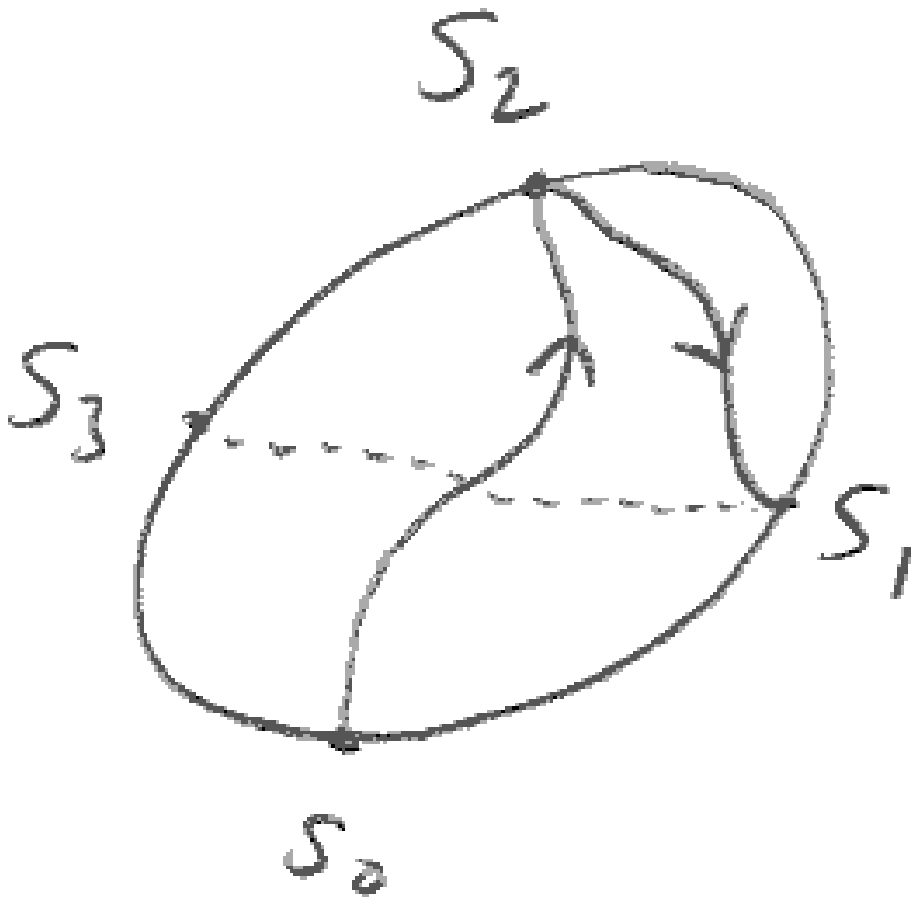}{Subarcs $\gamma(t_1)\gamma(s_2)$ and $\gamma(s_1)\gamma(s_3)$
intersect.}{.4}
In both cases, the subarcs $\gamma(t_1)\gamma(s_2)$ and $\gamma(s_1)\gamma(s_3)$
intersect (see \figref{crossing}), a contradiction. Therefore $T$ must be
simply increasing and then decreasing, not zigzag, on $[0,2\pi]$.
\end{proof}

\section{Special pairs of support lines}
Mainly, we wish to find a pair of support lines with prescribed angle difference
that touch $\gamma(s_1)$ and $\gamma(s_3)$ by one line and $\gamma(s_2)$
by the other line such that $s_1<s_2<s_3$ (see \figref{pair-lines}).
More specifically,
for a given $\delta<2\pi$, we find $\theta_1$ and $\theta_2$ such that $|\theta_1-\theta_2|$
is $\delta$
or $2\pi-\delta$ with $T(\theta_1)\supseteq \{s_1,s_3\}$ and $T(\theta_2)\ni s_2$ with
$s_1<s_2<s_3$.
\figs{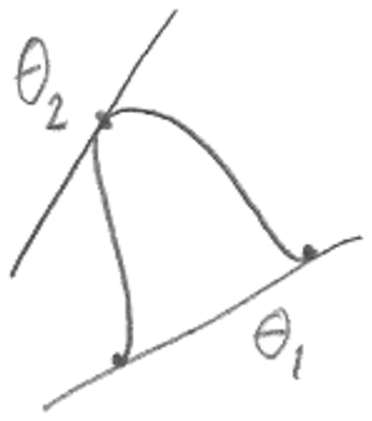}{A pair of lines with triple points.}{.7}
We will show the existence of this pair of support lines with triple points
on them.

Before we deal with the complicate multi-valued function $T$, we will practice
on similar functions and get similar results. First we try on continuous
functions.

\begin{lem}
Let $f:[0,2\pi]\to[0,1]$ be a continuous function with $f(0)=f(2\pi)=0$ and
$f(c)=1$ for some $c$ in $(0,2\pi)$. Suppose $f$ is strictly increasing on
$[0,c]$ and is strictly decreasing on $[c,2\pi]$ and $0<\delta<2\pi$.
Then there exists a unique $x$ such that $f(x)=f(x+\delta)$.
\end{lem}
\begin{proof}
Our plan is to find $y$ that $f^{-1}(y)$ contains $x$ and $x+\delta$. To
get such $x$ and $y$, we will use inverses of restrictions of $f$ as follows.
First note that $f$ is 1-1 and onto on $[0,c]$ and on $[c,2\pi]$.
Let $D=(f|_{[c,2\pi]})^{-1}-(f|_{[0,c]})^{-1}$. Since $D$ is strictly decreasing
and continuous
on $[0,1]$ starting from $2\pi$ down to 0, there is a unique $y$ that $D(y)=\delta$.
Equivalently there is a unique $x$ such that $f(x)=y=f(x+\delta)$.
\end{proof}

 Now we go back to the complicate function $T$. Note that
$\delta_1$ and $\delta_n$ are the exterior angles of the hull at $P_1$ and
$P_n$. From the graph, $\delta_1$ and $\delta_n$ are the widths of the minimum and maximum sets. From Lemma \ref{hor conn}, we have $\delta_1,\delta_n<\pi$.
Now we define at each $\theta$, the interval $I_\theta$ to be the closed interval $[\min T(\theta),\max T(\theta)]$ (illustratively and correspondingly
considered as the vertical segment $\{\theta\}\times I_\theta$ over $x=\theta$ in the graph.

\begin{thm}
If $\delta_n\le\delta<2\pi$, there exists a unique pair of support lines of angle
difference $\delta$ with triple points on them.
\end{thm}
\begin{proof}
\figs{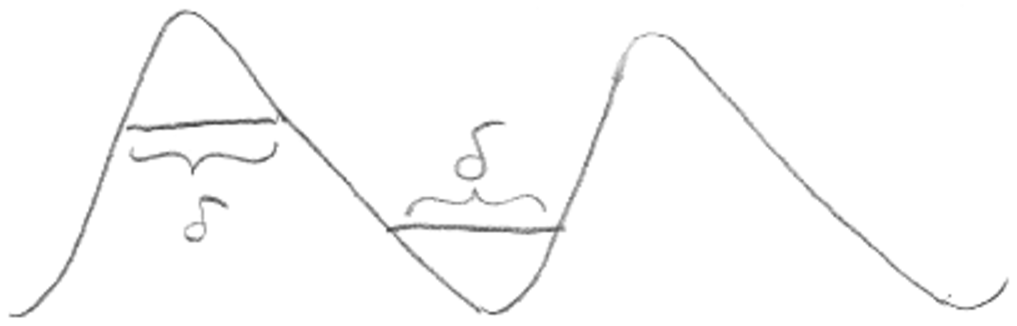}{Top and bottom of the graph of $T$.}{.7}
Our plan is to scan down from the top of the graph over the interval $[0,2\pi]$. Since $\delta\ge\delta_n$, we may find $s$ together with $\theta$
such that $s\in I_\theta$ and $s\in I_{\theta+\delta}$  (see \figref{top-bottom}). There the mountain-like
graph is $\delta$ wide on level $s$.
If $\delta=\delta_n$, we have $T^{-1}(t_n)$ in the form $[\alpha,\alpha+\delta_n]$
together with the arc in the situation as illustrated by \figref{delta1}.
\figs{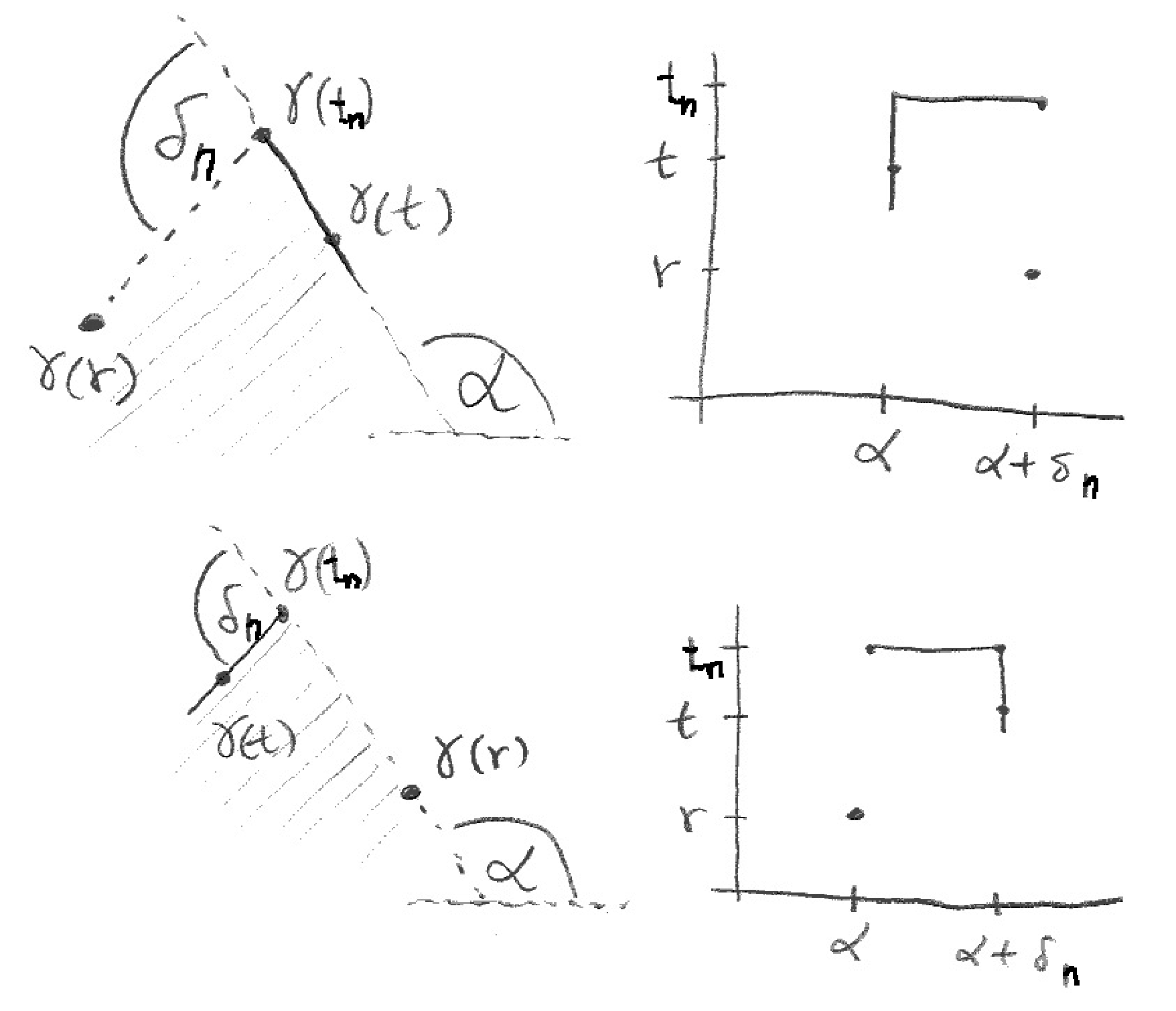}{The situation at $\gamma(t_n)=P_n$.}{.3}
Hence we have lines $L_\alpha$ and $L_{\alpha+\delta_n}$ and $r<t<t_n$ such
that $\gamma(t)$ is
on such line and $\gamma(r)$ and $\gamma(t_n)=P_n$ are on the other line.
Now we are in the case $\delta>\delta_n$.
First we will fill up the gaps where $T$ jumps by defining $T^|(\theta)=I(\theta)$.
Note that  for intervals $I$ and $J$, the substraction $I-J$ is simply $\{i-j|i\in
I$ and $j\in J\}$. Let $\theta_n$ be in $T^{-1}(t_n)$.
The filled, multi-valued function $T^|$ is onto and increasing on $[0,\theta_n]$
and is onto and decreasing on $[\theta_n,2\pi]$ (looks like a step pyramid). Hence the inverses of both
restrictions are multi-valued functions that are increasing and decreasing
respectively. Let $D=(T^||_{[\theta_n,2\pi]})^{-1}-(T^||_{[0,\theta_n]})^{-1}$. Note that $D$ is a stepping down function with every step filled.
Now, since $D$ is decreasing (together with the single-valued functions $\max D$ and $\min D$),
there is $s$ that $D(s)$ contains $\delta$. Equivalently there is $\theta$ that
$s\in T^|(\theta)\cap T^|(\theta+\delta)=I(\theta)\cap I(\theta+\delta)$.
Suppose both $I(\theta)$ and $I(\theta+\delta)$ degenerate. Then $T(\theta)=T(\theta+\delta)=\{s\}$.
Since $\delta>\delta_n$, we have $s<t_n$.
By Lemma \ref{hor conn}, the interval $[\theta,\theta+\delta]$ is a subset of $T^{-1}(s)$ and contains $\theta_n$ where $T$ takes value $t_n>s$, a contradiction.
Hence $I(\theta)$ or $I(\theta+\delta)$ does not degenerate.
Then one endpoint of such interval is in the other nondegenerated interval.
Suppose for the first case that $\min I(\theta)\in I(\theta+\delta)$. Let $s_1=\min I(\theta+\delta)$, $s_2=\min
I(\theta)$ and $s_3=\max I(\theta+\delta)$.
Thus $s_1\le s_2\le s_3$ and $s_1<s_3$. We have $\gamma(s_2)$ on $L_\theta$ and $\gamma(s_1)$, $\gamma(s_3)$ on $L_{\theta+\delta}$.
If $s_1=s_2$, $T^{-1}(s_1)$ is at least $\delta$ wide, a contradiction.
Hence $s_1<s_2<s_3$. Next we will show the uniqueness of $\theta$. Suppose
$L_{\theta'}$ and $L_{\theta'+\delta}$ have such triple points on them with
$\theta'\ne\theta$. Hence $I_{\theta'}\cap I_{\theta'+\delta}$ contains some
$s'$. We may assume $\theta<\theta'$. Since $\theta<\theta'\le\max T^{-1}(t_n)$, $T(\theta)\le T(\theta')$. Similarly, since $\min T^{-1}(t_n)\le\theta+\delta<\theta'+\delta$,
$T(\theta+\delta)\ge T(\theta'+\delta)$. We must have $s=s'$. Thus $T$ is constant over
$[\theta,\theta']$ and over $[\theta+\delta,\theta'+\delta]$. By Lemma \ref{hor
conn}, $T^{-1}(t_n)$ is longer than $\delta$, a contradiction.
The other cases can be treated similarly.
\end{proof}

Similarly we have the following theorem.

\begin{thm}
If $\delta_1\le\delta<2\pi$, there exist a unique pair of support lines of angle
difference $2\pi-\delta$ with triple points on them.
\end{thm}
\begin{proof}
Now we scan up from the bottom of the graph over the interval $(0,4\pi)$.
Precisely, as the graph looks like ``M", look at the middle part. More precisely,
look over the interval $[\theta,\theta+2\pi]$ where $T$ attains maximum at
$\theta$.
\end{proof}

\begin{cor}
Both previous theorems give different pairs of support line for $\delta\ne\pi$.
But when $\delta=\pi$, the 2 pairs are identical.
\end{cor}
\begin{proof}
Clear.
\end{proof}

\end{document}